\documentclass[10pt]{article}

\usepackage{hyperref}

\usepackage{amsmath}
\usepackage{graphicx}
\usepackage{color}
\usepackage{forest}
\usepackage{amsfonts}
\usepackage{lscape}
\usepackage{amsthm}

\newtheorem{theo}{Theorem}
\title{Memory-Usage Advantageous\\ Block Recursive Matrix Inverse}
\author{Iria C. S. Cosme,\\
Isaac F. Fernandes,\\
Jo\~ao L. de Carvalho,\\
Samuel Xavier-de-Souza}
\date{}

\begin{document}


\maketitle

\begin{abstract}
The inversion of extremely high order matrices has been a challenging task because of the limited processing and memory capacity of conventional computers. In a scenario in which the data does not fit in memory, it is worth to consider exchanging less memory usage for more processing time in order to enable the computation of the inverse which otherwise would be prohibitive. We propose a new algorithm to compute the inverse of block partitioned matrices with a reduced memory footprint. The algorithm works recursively to invert one block of a $k \times k$ block matrix $M$, with $k \geq 2$, based on the successive splitting of $M$. It computes one block of the inverse at a time, in order to limit memory usage during the entire processing. Experimental results show that, despite increasing computational complexity, matrices that otherwise would exceed the memory-usage limit can be inverted using this technique.

\end{abstract}


\section{Introduction}

Matrix inversion is a computation task necessary in many scientific applications, such as signal processing, complex network analysis, statistics \cite{mccullagh1989generalized} and some eigenvalue-related problems~\cite{byers1987solving}, to name just a few. There are some commonly available matrix inversion algorithms for nonsingular matrices, like Gaussian elimination, Gauss-Jordan, LU Decomposition and Cholesky decomposition. The majority of these algorithms are computationally intensive in use of memory and processor. For example, computing the inverse of a $n \times n$ matrix with the Gauss-Jordan method has computational complexity of $O(n^{3})$ and memory storage complexity of $O(n^{2})$. This can forbid the applicability of such methods for large-scale matrices, mainly, because data may simply not fit in memory.

Working with big data is a situation becoming increasingly common in today's world due to advances in sensor and communication technologies and to the evolution of digital data. This has become a challenging task because the handled data exceed conventional sizes and are often collected at speed greater than the processing and memory capacity of conventional computers can handle. In classification (or regression) problems~\cite{an2007fast}, for example, inverting matrices of extremely high order can lead to exceeding the computer memory capacity just to store its inverse. Many algorithms were developed with the objective to speed up the processing of large matrix computations. In the sequel, we list a few of those efforts regarding block and recursive algorithms.

The use of block partitioned matrices is commonly used to cut down processing time of matrix computations. Block matrices may occur naturally due to the ordering of the equations and the variables in a wide variety of scientific and engineering applications, such as in the incompressible Navier-Stokes equations \cite{elman2008taxonomy}, mixed finite elements approximation of elliptic partial differential equations~\cite{brezzi1991mixed}, optimal control~\cite{betts2001practical}, electrical networks~\cite{bjorck1996numerical} and the Strassen's Algorithm~\cite{strassen1969gaussian} for fast matrix multiplication. Algorithms that manipulate matrices at the block level are often more efficient because they are more abundant in level-3 operations~\cite{golub2013matrix}, and thus, can be implemented recursively. 

There are many related papers on the inverse of block matrices. In~\cite{lu2002inverses}, the authors provide inverse formulae for $2 \times 2$ block matrices applied in block triangular matrices and various structured matrices such as Hamiltonian, per-Hermitian and centro-Hermitian matrices. The inversion of block circulant matrices has been extensively investigated in~\cite{baker1988recursive} and~\cite{vescovo1997inversion}. Also, in \cite{karimi2014block}, the authors suggest the use block preconditioner for the block partitioned matrices.

Likewise, recursive algorithms have been applied in~\cite{baker1988recursive},~\cite{madsen1983cyclic} and~\cite{tsitsas2007recursive} for the inversion of particular cases of matrices, such as circulant-structured matrices. In~\cite{madsen1983cyclic} a recursive method is proposed for the LU decomposition of a real symmetric circulant matrix. In~\cite{baker1988recursive}, a recursive algorithm is applied to calculate of the first block row of the inverse of a block circulant matrix with circulant blocks. And, finally, in~\cite{tsitsas2007recursive}, the authors propose a recursive algorithm based on the inversion of $k \times k$ block matrices for cases of matrices with circulant blocks based on the previous diagonalization of each circulant block. 

In addition, in \cite{tsitsas2010DTT}, the authors propose an efficient method for the inversion of matrices with $U$-diagonalizable blocks (being $U$ a fixed unitary matrix) by utilizing the $U$-diagonalization of each block and subsequently a similarity transformation procedure. This approach allows getting the inverse of matrices with $U$-diagonalizable blocks without having to assume the invertibility of the blocks involved in the procedure, provided certain conditions met.

Despite these numerous efforts, the inversion of large dense matrices is still challenging for large datasets. In a scenario in which the data does not fit in the memory, exchanging less memory usage for more processing time could be an alternative to enable the computation of the inverse which otherwise would be prohibitive.

In this paper, motivated by the preceding considerations, we introduce a recursive method for the inversion of a $k \times k$ block matrix $M \in \mathbb{R}^{m \times m}$ with square blocks of order $b$. The basic idea of this algorithm, called Block Recursive Inversion (BRI), lies in the determination of one block of the inverse of the matrix at a time. The method is based on the successive splitting of the original matrix into 4 square matrices of an inferior order, called frames. For this, it is considered two stages, namely the forward recursive procedure and backward recursive procedure. The forward recursive procedure terminates after $k-2$ steps when the resulting frames have $2 \times 2$ blocks. Thereafter, in the backward recursive procedure, for each 4 frames generated, operations are carried out reducing them to a single block. Differently from those proposed in~\cite{baker1988recursive,madsen1983cyclic,tsitsas2007recursive}, our recursive algorithm may be applied for inverting any $M \in \mathbb{R}^{m \times m}$, provided that $M$ is nonsingular and that all submatrices that need to be inverted in the recurrent procedure are also nonsingular.

The recursive algorithm proposed in \cite{tsitsas2007recursive} presents lower computational complexity than the BRI in inversion processes for the cases of matrices with circulant blocks. However, for the cases that the data do not fit in memory, using the BRI might be more adequate since it has lower memory storage complexity.

Besides requiring a much smaller memory footprint to work, having an algorithm that computes only parts of the inverse at a time may be useful for some applications. In \cite{an2007fast}, for example, only the diagonal blocks of the kernel matrix of the Least Squares Support Vector Machine (LS-SVM) are used to compute the predicted labels in the cross-validation algorithm. Thus, it is not strictly necessary to compute and, in fact, store in memory the entire inverse of the kernel matrix.

This paper is organized as follows. In Section \ref{sec:preliminares}, we introduce some notations and summarize the main definitions about the inversion of $2 \times 2$ block matrices using Schur Complement. In Section \ref{sec:recursive_algorithm}, we present the proposed recursive algorithm. In Section \ref{sec:inv4x4}, we demonstrate a representative example of the recursive inversion of $4 \times 4$ block matrices. So, we introduce details of the implementation in Section \ref{sec:implementation}. In Section \ref{sec:complexity}, we present the complexity analysis of the proposed algorithm, including an investigation of the cost of memory storage computational. In Section \ref{sec:results}, we describe the experiments as well as the results obtained. Finally, in Section \ref{sec:conclusion}, we present the conclusions and future works this research.

\section{Preliminaries} \label{sec:preliminares}

In this section, we begin with some basic notations which are frequently used in the sequel.

Let $M \in \mathbb{R}^{m \times m}$ be a nonsingular matrix composed of square blocks $M_{\alpha \beta}$ of the same order
\begin{equation} \label{eq:blockmatricesnxn}
M = \left[\begin{array}{ccc}
M_{11} & \cdots & M_{1k} \\
\vdots & \ddots & \vdots \\
M_{k1} & \cdots & M_{kk} 
\end{array}\right],
\end{equation}
where $M_{\alpha \beta}$ designates the $(\alpha , \beta)$ block. With this notation, block $M_{\alpha \beta}$ has dimension $b \times b$, with $b=\dfrac{m}{k}$, and $M = (M_{\alpha \beta})$ is a $k\times k$ block matrix.

Assume that $M^{-1}$ is the inverse matrix of $M$. So,
\begin{equation} \label{eq:inverseblockmatricesnxn}
M^{-1} = \left[\begin{array}{ccc}
N_{11} & \cdots & N_{1k} \\
\vdots & \ddots & \vdots \\
N_{k1} & \cdots & N_{kk} 
\end{array}\right]. 
\end{equation}

Assigning $k=2$ results in a $2 \times 2$ block matrix as 
\begin{equation} \label{eq:blockmatrices2x2}
M = \left[\begin{array}{cc}
A & B \\
C & D 
\end{array}\right],  
\end{equation}
where $A:=M_{11}$, $B:=M_{12}$, $C:=M_{21}$, and $D:=M_{22}$; and $M^{-1}$ is its inverse matrix as
\begin{equation} \label{eq:inv_blockmatrices2x2}
M^{-1} = \left[\begin{array}{cc}
E & F \\
G & H 
\end{array}\right],  
\end{equation}
where $E:=N_{11}$, $F:=N_{12}$, $G:=N_{21}$, and $H:=N_{22}$.

If $A$ is nonsingular, the Schur complement of $M$ with respect to $A$\cite{zhang2005schur}, denoted by $(M/A)$, is defined by 
\begin{equation} \label{eq:SchurofA}
(M/A) = D-CA^{-1}B.
\end{equation}

Let us remark that it is possible to define similarly the following Schur complements 
\begin{equation} \label{eq:SchurofB}
(M/B) = C-DB^{-1}A,
\end{equation}
\begin{equation} \label{eq:SchurofC}
(M/C) =B-AC^{-1}D,  \quad \mbox{and}
\end{equation}
\begin{equation} \label{eq:SchurofD}
(M/D) = A-BD^{-1}C,
\end{equation}
provided that the matrices $B$, $C$, and $D$ are nonsingular only in (\ref{eq:SchurofB}), (\ref{eq:SchurofC}), and (\ref{eq:SchurofD}), respectively \cite{brezinski1988other}.

Considering the $D$ block, if both $M$ and $D$ in (\ref{eq:blockmatrices2x2}) are nonsingular, then $(M/D) = A-BD^{-1}C$ is nonsingular, too, and $M$ can be decomposed as
\begin{equation} \label{eq:decomposedM}
M = \left[\begin{array}{cc}
I & BD^{-1} \\
0 & I
\end{array}\right]
\left[\begin{array}{cc}
(M/D) & 0 \\
0 & D
\end{array}\right]
\left[\begin{array}{cc}
I & 0 \\
D^{-1}C & I
\end{array}\right], 
\end{equation}
where $I$ is the identity matrix. In this case, the inverse of $M$ can be written as 
\begin{eqnarray} \label{eq:invSchur}
M^{-1} & = & \left[\begin{array}{cc} \nonumber
\scriptstyle I & \scriptstyle 0 \\
\scriptstyle -D^{-1}C & \scriptstyle I
\end{array}\right]
\left[\begin{array}{cc}
\scriptstyle (M/D)^{-1} & \scriptstyle 0 \\
\scriptstyle 0 & \scriptstyle D^{-1}
\end{array}\right]
\left[\begin{array}{cc}
\scriptstyle I & \scriptstyle -BD^{-1} \\
\scriptstyle 0 & \scriptstyle I
\end{array}\right] \\ 
& = & \left[\begin{array}{cc}
\scriptstyle (M/D)^{-1} & \scriptstyle -(M/D)^{-1}BD^{-1} \\ 
\scriptstyle  -D^{-1}C(M/D)^{-1} & \scriptstyle  D^{-1}+D^{-1}C(M/D)^{-1}BD^{-1}
\end{array}\right]. 
\end{eqnarray}

The formulation in (\ref{eq:invSchur}) is well known and has extensively been used in dealing with inverses of block matrices \cite{noble1988applied}.

\section{Block Recursive Inversion algorithm} \label{sec:recursive_algorithm}

It is worth noting that, according to (\ref{eq:invSchur}), to get $E$ just calculate the inverse of Schur complement of $D$ in $M$. 

Following this premise, it was observed that to obtain the block $N_{11}$ of the inverse of the $k \times k$ block matrix $M$, with $k>2$, one should successively split $M$ into $4$ submatrices of the same order, called frames, until the resulting frames are $2\times 2$ block matrices. Then, for each resulting $2 \times 2$ frame, the Schur complement should be applied in the opposite direction of the recursion, reducing the frame to a single block, which should be combined with the other 3 blocks of the recursion branch to form $2 \times 2$ frames that should be reduced to single blocks, successively. Finally, $N_{11}$ is obtained by calculating the inverse of the last block resulting from the reverse recursion process. For this, all submatrices which are required to be inverted in the proposed algorithm must be nonsingular.

For a better understanding, the algorithm will be explained considering two stages, namely the \textit{forward recursive procedure} and \textit{backward recursive procedure}. For this, consider (\ref{eq:blockmatricesnxn}) as the input block matrix.

\subsection{Forward recursive procedure} \label{subsec:Forward}

\textbf{Step 1:} Split the input, a $k \times k$ block matrix $M$, into four $(k - 1) \times (k-1)$ block matrices $M \rangle_A$, $M \rangle_B$, $M \rangle_C$, and $M \rangle_D$, called frames, excluding one of the block rows of $M$ and one of the block columns of $M$ as follows.

The frame $M \rangle_A$ results from the removal of the lowermost block row and rightmost block column of $M$; $M \rangle_B$ results from the removal of the lowermost block row and leftmost block column of $M$; $M \rangle_C$ results from the removal of the uppermost block row and rightmost block column of $M$ and, finally, $M \rangle_D$ results from the removal of the uppermost block row and leftmost block column of $M$, as shown in (\ref{eq:F_k-1}).

\begin{eqnarray} \label{eq:F_k-1}
M \rangle_{A}  = \left[\begin{array}{ccc}
M_{11} & \cdots & M_{1(k-1)} \\
\vdots & \ddots & \vdots \\
M_{(k-1)1} &  \cdots & M_{(k-1)(k-1)} 
\end{array}\right], \\ \nonumber
M \rangle_B = \left[\begin{array}{ccc}
M_{12} & \cdots & M_{1k} \\
\vdots & \ddots & \vdots \\
M_{(k-1)2} & \cdots & M_{(k-1)k} 
\end{array}\right], \\ \nonumber
M \rangle_{C} = \left[\begin{array}{ccc}
M_{21} & \cdots & M_{2(k-1)} \\
\vdots & \ddots & \vdots \\
M_{k1} & \cdots & M_{k(k-1)} 
\end{array}\right], \\  \nonumber
M \rangle_{D} = \left[\begin{array}{ccc}
M_{22} & \cdots & M_{2k} \\
\vdots & \ddots & \vdots \\
M_{k2} & \cdots & M_{kk} 
\end{array}\right].
\end{eqnarray} 
\\

\textbf{Step 2:} Exchange block rows and/or block columns until block $M_{22}$ reaches its original place, namely, second block row and second block column. In this case, in the Frame $M \rangle_A$, no exchange occurs. In $M \rangle_B$, it is necessary to exchange the two leftmost block columns. In $M \rangle_C$, it is necessary to exchange the two uppermost block rows. And, finally, in $M \rangle_D$, simply exchange the two leftmost block columns and the two uppermost block rows. So, applying these exchanges results in 
\begin{eqnarray} \label{eq:F_k-1_ChangeM22}
M \rangle_A & = & \left[\begin{array}{cccc} \nonumber
M_{11} & M_{12} & \cdots & M_{1(k-1)} \\
M_{21} & M_{22} & \cdots & M_{2(k-1)} \\
M_{31} & M_{32} & \cdots & M_{3(k-1)} \\
\vdots & \ddots & \vdots \\
M_{(k-1)1} & M_{(k-1)2} & \cdots & M_{(k-1)(k-1)} 
\end{array}\right], \\ 
M \rangle_B & = & \left[\begin{array}{cccc}
M_{13} & M_{12} & \cdots & M_{1k} \\
M_{23} & M_{22} & \cdots & M_{2k} \\
M_{33} & M_{32} & \cdots & M_{3k} \\
\vdots & \vdots & \ddots & \vdots \\
M_{(k-1)3} & M_{(k-1)2} & \cdots & M_{(k-1)k} 
\end{array}\right], \\ \nonumber
M \rangle_C & = & \left[\begin{array}{cccc}
M_{31} & M_{32} & \cdots & M_{3(k-1)} \\
M_{21} & M_{22} & \cdots & M_{2(k-1)} \\
M_{41} & M_{42} & \cdots & M_{4(k-1)} \\
\vdots & \vdots & \ddots & \vdots \\
M_{k1} & M_{k2} & \cdots & M_{k(k-1)} 
\end{array}\right], {\rm~and}  \\ \nonumber
M \rangle_D & = & \left[\begin{array}{cccc}
M_{33} & M_{32} & \cdots & M_{3k} \\
M_{23} & M_{22} & \cdots & M_{2k} \\
M_{43} & M_{42} & \cdots & M_{4k} \\
\vdots & \vdots & \ddots & \vdots \\
M_{k3} & M_{k2} & \cdots & M_{kk} 
\end{array}\right].
\end{eqnarray} 

Next, split each one of the frames in (\ref{eq:F_k-1_ChangeM22}), $M \rangle_A$, $M \rangle_B$, $M \rangle_C$, and $M \rangle_D$ into four $(k-2) \times (k-2)$ frames, excluding one of its block rows and one of its block columns, as instructed in Step 1. For example, splitting the frame $M \rangle_D$ results in
\begin{eqnarray} \label{eq:F_k-2}
M \rangle_D \rangle_A & = & \left[\begin{array}{cccc}  \nonumber
M_{33} & M_{32} & \cdots & M_{3(k-1)} \\
M_{23} & M_{22} & \cdots & M_{2(k-1)} \\
M_{43} & M_{42} & \cdots & M_{4(k-1)} \\
\vdots & \vdots & \ddots & \vdots \\
M_{(k-1)3} & M_{(k-1)2} & \cdots & M_{(k-1)(k-1)} 
\end{array}\right], \\ 
M \rangle_D \rangle_B & = & \left[\begin{array}{cccc}
M_{32} & M_{34} & \cdots & M_{3k} \\
M_{22} & M_{24} & \cdots & M_{2k} \\
M_{42} & M_{44} & \cdots & M_{4k} \\
\vdots & \vdots & \ddots & \vdots \\
M_{(k-1)4} & M_{(k-1)2} & \cdots & M_{(k-1)k}
\end{array}\right], \\ \nonumber
M \rangle_D \rangle_C & = & \left[\begin{array}{cccc}
M_{23} & M_{22} & \cdots & M_{2(k-1)} \\
M_{43} & M_{42} & \cdots & M_{4(k-1)} \\
M_{53} & M_{52} & \cdots & M_{5(k-1)} \\
\vdots & \vdots & \ddots & \vdots \\
M_{k3} & M_{k2} & \cdots & M_{k(k-1)}  
\end{array}\right], {\rm ~and} \\  \nonumber
M \rangle_D \rangle_D & = & \left[\begin{array}{cccc}
M_{22} & M_{24} & \cdots & M_{2k} \\
M_{42} & M_{44} & \cdots & M_{4k} \\
M_{52} & M_{54} & \cdots & M_{5k} \\
\vdots & \vdots & \ddots & \vdots \\
M_{k2} & M_{k4} & \cdots & M_{kk}  
\end{array}\right].
\end{eqnarray} 
\\

\textbf{Step $i$:} For each of the $(k-(i-1)) \times (k-(i-1))$ frames resulting from the previous step (Step $i-1$),  $M \rangle^{i-1}_{A}$, $M \rangle^{i-1}_{B}$, $M \rangle^{i-1}_{C}$ and $M \rangle^{i-1}_{D}$, make the permutation of block rows and/or blocks columns and then generate more four  $(k-i)\times(k-i)$ frames, $M \rangle^{i}_{A}$, $M \rangle^{i}_{B}$, $M \rangle^{i}_{C}$ and $M \rangle^{i}_{D}$,  excluding one of its block rows and one of its block columns in an analogous manner to what was done in Step 2 to the frames resulting from Step 1. Repeat Step $i$ until $i=k-2$.

The superscript number to the "$\rangle$" symbol indicates the amount of these existing symbols, including those not represented. Thus, it denotes $M \rangle_x^y$ with $y \in \mathbb{N}^*$ and $x \in \{A, B, C, D\}$.

\subsection{Backward recursive procedure} \label{subsec:Backward}

\textbf{Step 1:} For each of the resulting $2 \times 2$ frames from the Step $k-2$ in the forward recursive procedure, in Subsection \ref{subsec:Forward}, $M \rangle^{k-2}_A$, $M \rangle^{k-2}_{B}$, $M \rangle^{k-2}_{C}$, and $M \rangle^{k-2}_{D}$, compute the Schur complement of $M_{22}$ to generate the blocks: $\langle M \rangle^{k-2}_A$, $\langle M \rangle^{k-2}_{B}$, $\langle M \rangle^{k-2}_{C}$, and $\langle M \rangle^{k-2}_{D}$, using (\ref{eq:SchurofD}), (\ref{eq:SchurofC}), (\ref{eq:SchurofB}), and (\ref{eq:SchurofA}), in this order. The ''$\langle$'' symbol indicates that the Schur complement operation was done on the respective frame.  
\\

\textbf{Step 2:} Assemble $\langle M \rangle^{k-3}_A \rangle$, $\langle M \rangle^{k-3}_B \rangle$, $\langle M \rangle^{k-3}_C \rangle$, and $\langle M \rangle^{k-3}_D \rangle$, joining the four frames that were originated by each in the forward recursive procedure previously reported in Subsection \ref{subsec:Forward}, in the following way:
\begin{equation} \label{eq:BlockMatrix_Phi_A}
\langle M \rangle^{k-3}_A\rangle = \left[\begin{array}{cc}
\langle M \rangle^{k-3}_A\rangle_A  & \langle M \rangle^{k-3}_A\rangle_B \\
\langle M \rangle^{k-3}_A\rangle_C  & \langle M \rangle^{k-3}_A\rangle_D 
\end{array}\right],
\end{equation}

\begin{equation} \label{eq:BlockMatrix_Phi_B}
\langle M \rangle^{k-3}_B\rangle = \left[\begin{array}{cc}
\langle M \rangle^{k-3}_B\rangle_A  & \langle M \rangle^{k-3}_B\rangle_B \\
\langle M \rangle^{k-3}_B\rangle_C  & \langle M \rangle^{k-3}_B\rangle_D 
\end{array}\right],
\end{equation}

\begin{equation} \label{eq:BlockMatrix_Phi_C}
\langle M \rangle^{k-3}_C\rangle = \left[\begin{array}{cc}
\langle M \rangle^{k-3}_C\rangle_A  & \langle M \rangle^{k-3}_C\rangle_B \\
\langle M \rangle^{k-3}_C\rangle_C  & \langle M \rangle^{k-3}_C\rangle_D 
\end{array}\right],
\end{equation}

\begin{equation} \label{eq:BlockMatrix_Phi_D}
\langle M \rangle^{k-3}_D\rangle = \left[\begin{array}{cc}
\langle M \rangle^{k-3}_D\rangle_A  & \langle M \rangle^{k-3}_D\rangle_B \\
\langle M \rangle^{k-3}_D\rangle_C  & \langle M \rangle^{k-3}_D\rangle_D 
\end{array}\right]. 
\end{equation}
Then, for each of these $2 \times 2$ block matrices, calculate the Schur complement applying (\ref{eq:SchurofD}), (\ref{eq:SchurofC}), (\ref{eq:SchurofB}), and (\ref{eq:SchurofA}), in this order, generating the blocks: $\langle \langle M \rangle^{k-3}_A\rangle$, $\langle \langle M \rangle^{k-3}_B\rangle$, $\langle \langle M \rangle^{k-3}_C\rangle$, and $\langle \langle M \rangle^{k-3}_D\rangle$.
\\

\textbf{Step $i$:} Considering each four branches of the recursion, repeat the previous step generating frames $\langle^{i} M \rangle_{A}^{k-(i-1)} \rangle^{i-1}$, $\langle^{i} M \rangle_{B}^{k-(i-1)} \rangle^{i-1}$, $\langle^{i} M \rangle_{C}^{k-(i-1)} \rangle^{i-1}$, and $\langle^{i} M \rangle_{D}^{k-(i-1)} \rangle^{i-1}$, until $i=k-2$ and thus get $\langle^{k-2} M \rangle_A \rangle^{k-3}$, $\langle^{k-2} M \rangle_B \rangle^{k-3}$, $\langle^{k-2} M \rangle_C \rangle^{k-3}$, and $\langle^{k-2} M \rangle_D \rangle^{k-3}$.
\\

\textbf{Step $k-1$:} Assemble $\langle^{k-2}  M \rangle^{k-2}$ from $\langle^{k-2}  M \rangle_A \rangle^{k-3}$, $\langle^{k-2}  M \rangle_B \rangle^{k-3}$, $\langle^{k-2}  M \rangle_C \rangle^{k-3}$, and $\langle^{k-2}  M \rangle_D \rangle^{k-3}$, in the following way:

\begin{equation} \label{eq:BlockMatrix_Phi_M}
\langle^{k-2}  M \rangle^{k-2} = \left[\begin{array}{cc}
\langle^{k-2}  M \rangle_A \rangle^{k-3} & \langle^{k-2}  M \rangle_B \rangle^{k-3} \\
\langle^{k-2}  M \rangle_C \rangle^{k-3} & \langle^{k-2}  M \rangle_D \rangle^{k-3}
\end{array}\right].
\end{equation}
And, finally, calculate the Schur complement of $\langle^{k-2}  M \rangle^{k-2}$ with respect to $\langle^{k-2}  M \rangle_D \rangle^{k-3}$ using (\ref{eq:SchurofD}). Thus, the inverse of the matrix corresponding to $N_{11}$ corresponds to 
\begin{eqnarray} \label{eq:blockmatrix_NxN}
N_{11} & = & (\langle^{k-2}  M \rangle^{k-2} / \langle^{k-2}  M \rangle_D \rangle^{k-3})^{-1}.
\end{eqnarray} \\

It is worth noting that with a suitable permutation of rows and columns, we can position any block in the upper left-hand corner of $M$ and get the corresponding inverse to this block with respect to $M$.

Since the dimension $m$ and the number of blocks $k$ of $M$ are arbitrary, for the cases of the order $b$ of each block being $b=\frac{m}{k}$, with $b \notin \mathbb{N}$, it is possible to consider a matrix $\Gamma$ as an augmented matrix of $M$ as follows.

Let $m$ be the order of a square matrix $M \in \mathbb{R}^{m \times m}$, $k$ be the arbitrary number of blocks that we want to partition the input matrix of the BRI algorithm and $l$ the minimum integer with $\frac{m+l}{k} \in \mathbb{N}$. The augmented matrix of $M$ is
\begin{equation} \label{eq:M_augmented}
\Gamma = \left[\begin{array}{cc}
M & 0 \\
0^{T} & I 
\end{array}\right],  
\end{equation}
with $\Gamma = M$ for $\frac{m}{k} \in \mathbb{N}$, where $0$ is the $(m \times l)$ zero matrix and $I$ the identify matrix of order $l$. This way, the matrix $\Gamma$ may be partitioned in $k$ blocks and the inverse $M^{-1}$ of $M$ will be derived by applying the BRI algorithm to the matrix $\Gamma$.

Thus, we have that the inverse of $\Gamma$ is:
\begin{equation} \label{eq:inverse_M_augmented}
\Gamma^{-1} = \left[\begin{array}{cc}
M^{-1} & 0 \\
0^{T} & I 
\end{array}\right]. 
\end{equation}

\section{An example: inverse of $4 \times 4$ block matrices} \label{sec:inv4x4}

In order to clarify the operation of the proposed algorithm of the inversion of $k \times k$ block matrices, this section presents the process of inverting of a $4 \times 4$ block matrix.

The basic idea of the recursive algorithm, as shown in Section \ref{sec:recursive_algorithm}, for the inversion of $4 \times 4$ block matrices lies in the fact that in each step the involved matrices are split into four square matrices of the same order until getting to $2 \times 2$ block matrices.

Consider that a nonsingular $4b \times 4b$ matrix $M$ can be partitioned into $4 \times 4$ blocks of order $b$ as

\begin{equation} \label{eq:blockmatriz4x4}
M = \left[\begin{array}{cccc}
M_{11} & M_{12} & M_{13} & M_{14}\\
M_{21} & M_{22} & M_{23} & M_{24}\\
M_{31} & M_{32} & M_{33} & M_{34}\\
M_{41} & M_{42} & M_{43} & M_{44}\\
\end{array}\right]
\end{equation}
and $M^{-1}$ is its inverse matrix:

\begin{equation} \label{eq:invblockmatrix4x4}
M^{-1} = \left[\begin{array}{cccc}
N_{11} & N_{12} & N_{13} & N_{14}\\
N_{21} & N_{22} & N_{23} & N_{24}\\
N_{31} & N_{32} & N_{33} & N_{34}\\
N_{41} & N_{42} & N_{43} & N_{44}\\
\end{array}\right].
\end{equation}

Now, consider four square matrices of order $3b$ generated from $M$ by the process of excluding one of its block rows and one of its block columns, as shown in Step 1 of Subsection \ref{subsec:Forward}. After positioning the block $M_{22}$ in its original position in $M$ (Step 2), these frames are as follows:

\begin{eqnarray} \label{eq:frames4x4Matrix_InverterLouC}
M \rangle_A = \left[\begin{array}{ccc}
M_{11} & M_{12} & M_{13}\\
M_{21} & M_{22} & M_{23}\\
M_{31} & M_{32} & M_{33}
\end{array}\right], \quad
M \rangle_B  = \left[\begin{array}{ccc}
M_{13} & M_{12} & M_{14}\\
M_{23} & M_{22} & M_{24}\\
M_{33} & M_{32} & M_{34}
\end{array}\right], \\ \nonumber
M \rangle_C  = \left[\begin{array}{ccc}
M_{31} & M_{32} & M_{33}\\
M_{21} & M_{22} & M_{23}\\
M_{41} & M_{42} & M_{43}
\end{array}\right], {\rm ~and} \quad
M \rangle_D  = \left[\begin{array}{ccc}
M_{33} & M_{32} & M_{34}\\
M_{23} & M_{22} & M_{24}\\
M_{43} & M_{42} & M_{44}
\end{array}\right].
\end{eqnarray} 

Thus, recursively, each of the frames $M \rangle_A$, $M \rangle_B$, $M \rangle_C$ and $M \rangle_D$ will be divided into four $2 \times 2$ frames, applying the rules discussed in Subsection \ref{subsec:Forward} (Steps 2 and i). This process is illustrated in Figure \ref{fig:ForwardProcedure}.

Splitting, for example, the $M \rangle_B$  frame, the following frames are obtained:

\begin{eqnarray} \label{eq:frames2x2_of_FB3}
M \rangle_B \rangle_A = \left[\begin{array}{cc}
M_{13} & M_{12}\\
M_{23} & M_{22}
\end{array}\right], \quad
M \rangle_B \rangle_B = \left[\begin{array}{cc}
M_{12} & M_{14}\\
M_{22} & M_{24}
\end{array}\right], \\ \nonumber
M \rangle_B \rangle_C = \left[\begin{array}{cc}
M_{23} & M_{22}\\
M_{33} & M_{32}
\end{array}\right], {\rm ~and}\quad
M \rangle_B \rangle_D = \left[\begin{array}{cc}
M_{22} & M_{24}\\
M_{32} & M_{34}
\end{array}\right].
\end{eqnarray}

As the Step 1 of Subsection~\ref{subsec:Backward}, assuming $M_{22}$ as nonsingular and using (\ref{eq:SchurofD}), (\ref{eq:SchurofC}), (\ref{eq:SchurofB}) and (\ref{eq:SchurofA}), in that order, the Schur Complement of $M_{22}$ for each of the frames in (\ref{eq:frames2x2_of_FB3}) are calculated: 

\begin{equation} \label{eq:FA2/P22}
\langle M \rangle_B \rangle_A = (M \rangle_B \rangle_A/M_{22})= M_{13} - M_{12}M_{22}^{-1}M_{23};
\end{equation}
\begin{equation} \label{eq:FB2/P22}
\langle M \rangle_B \rangle_B = (M \rangle_B \rangle_B/M_{22})=M_{14} - M_{12}M_{22}^{-1}M_{24};
\end{equation}
\begin{equation} \label{eq:FC2/P22}
\langle M \rangle_B \rangle_C = (M \rangle_B \rangle_C/M_{22})=M_{33} - M_{32}M_{22}^{-1}M_{23};
\end{equation}
\begin{equation} \label{eq:FD2/P22}
\langle M \rangle_B \rangle_D = (M \rangle_B \rangle_D/M_{22})=M_{34} - M_{32}M_{22}^{-1}M_{24}.
\end{equation}

Running the Step 2 of the backward recursive procedure from the recursion algorithm presented in Subsection \ref{subsec:Backward}, produces:

\begin{equation} \label{eq:Phi_2_A}
\langle M \rangle_A \rangle = \left[\begin{array}{cc}
\langle M \rangle_A \rangle_A & \langle M \rangle_A \rangle_B \\
\langle M \rangle_A \rangle_C & \langle M \rangle_A \rangle_D \\
\end{array}\right],
\end{equation}

\begin{equation} \label{eq:Phi_2_B}
\langle M \rangle_B \rangle = \left[\begin{array}{cc}
\langle M \rangle_B \rangle_A & \langle M \rangle_B \rangle_B \\
\langle M \rangle_B \rangle_C & \langle M \rangle_B \rangle_D \\
\end{array}\right],
\end{equation}

\begin{equation} \label{eq:Phi_2_C}
\langle M \rangle_C \rangle = \left[\begin{array}{cc}
\langle M \rangle_C \rangle_A & \langle M \rangle_C \rangle_B \\
\langle M \rangle_C \rangle_C & \langle M \rangle_C \rangle_D \\
\end{array}\right], {\rm ~and}
\end{equation}

\begin{equation} \label{eq:Phi_2_D}
\langle M \rangle_D \rangle = \left[\begin{array}{cc}
\langle M \rangle_D \rangle_A & \langle M \rangle_D \rangle_B \\
\langle M \rangle_D \rangle_C & \langle M \rangle_D \rangle_D \\
\end{array}\right]
\end{equation}

Assuming $M_{22}$, $\langle M \rangle_A \rangle_D$, $\langle M \rangle_B \rangle_C$, $\langle M \rangle_C \rangle_B$, and $\langle M \rangle_D \rangle_A$ as nonsingular and using (\ref{eq:SchurofD}), (\ref{eq:SchurofC}), (\ref{eq:SchurofB}) and (\ref{eq:SchurofA}), in that order, the Schur Complement calculated for each of the frames are: 
\begin{eqnarray} \label{eq:Schur_Phi_2_A}
\langle \langle M \rangle_A \rangle & = & (\langle M \rangle_A \rangle / \langle M \rangle_A \rangle_D)\\
& = & \langle M \rangle_A \rangle_A - \langle M \rangle_A \rangle_B \langle M \rangle_A \rangle_D^{-1}\langle M \rangle_A \rangle_C   \nonumber
\end{eqnarray}
\begin{eqnarray} \label{eq:Schur_Phi_2_B}
\langle \langle M \rangle_B \rangle & = & (\langle M \rangle_B \rangle / \langle M \rangle_B \rangle_C) \\
& = & \langle M \rangle_B \rangle_B - \langle M \rangle_B \rangle_A\langle M \rangle_B \rangle_C^{-1}\langle M \rangle_B \rangle_D \nonumber
\end{eqnarray}
\begin{eqnarray} \label{eq:Schur_Phi_2_C}
\langle \langle M \rangle_C \rangle & = & (\langle M \rangle_C \rangle / \langle M \rangle_C \rangle_B) \\
& = & \langle M \rangle_C \rangle_C - \langle M \rangle_C \rangle_D\langle M \rangle_C \rangle_B^{-1}\langle M \rangle_C \rangle_A \nonumber
\end{eqnarray}
\begin{eqnarray} \label{eq:Schur_Phi_2_D}
\langle \langle M \rangle_A \rangle & = & (\langle M \rangle_D \rangle / \langle M \rangle_D \rangle_A) \\
& = & \langle M \rangle_D \rangle_D - \langle M \rangle_D \rangle_C\langle M \rangle_D \rangle_A^{-1}\langle M \rangle_D \rangle_B \nonumber
\end{eqnarray}

So, as the Step $k-1$ of Subsection \ref{subsec:Backward}, let $\langle \langle M \rangle \rangle$ be a block matrix generated by matrices resulting from (\ref{eq:Schur_Phi_2_A}), (\ref{eq:Schur_Phi_2_B}), (\ref{eq:Schur_Phi_2_C}), and (\ref{eq:Schur_Phi_2_D}), as shown in Figure \ref{fig:BackwardProcedure}.

\begin{equation} \label{eq:labelframes4x4}
\langle \langle M \rangle \rangle = \left[\begin{array}{cc}
\langle \langle M \rangle_A \rangle & \langle \langle M \rangle_B \rangle\\
\langle \langle M \rangle_C \rangle & \langle \langle M \rangle_D \rangle\\
\end{array}\right]
\end{equation}
Thus, $\langle \langle M \rangle \rangle$ is a $2 \times 2$ block matrix generated from the original $4 \times 4$ block matrix. So, in order to get $N_{11}$, simply get the inverse of Schur complement of $\langle \langle M \rangle_D \rangle$ in $\langle \langle M \rangle \rangle$, following the same formula used to get $E$, shown in the upper left-hand corner of  (\ref{eq:invSchur}), namely $(M/D)^{-1}=(A-BD^{-1}C)^{-1}$.

Applying this formula, we have 
\begin{eqnarray} \label{eq:blockmatrix4x4}
N_{11} & = & (\langle \langle M \rangle \rangle /\langle \langle M \rangle_D \rangle)^{-1}  \\
& = & (\langle \langle M \rangle_A \rangle - \langle \langle M \rangle_B \rangle \langle \langle M \rangle_D \rangle^{-1}\langle \langle M \rangle_C \rangle)^{-1}.\nonumber
\end{eqnarray}

\begin{figure}[!h]
\centering 
\includegraphics[width=\linewidth, height=35mm]{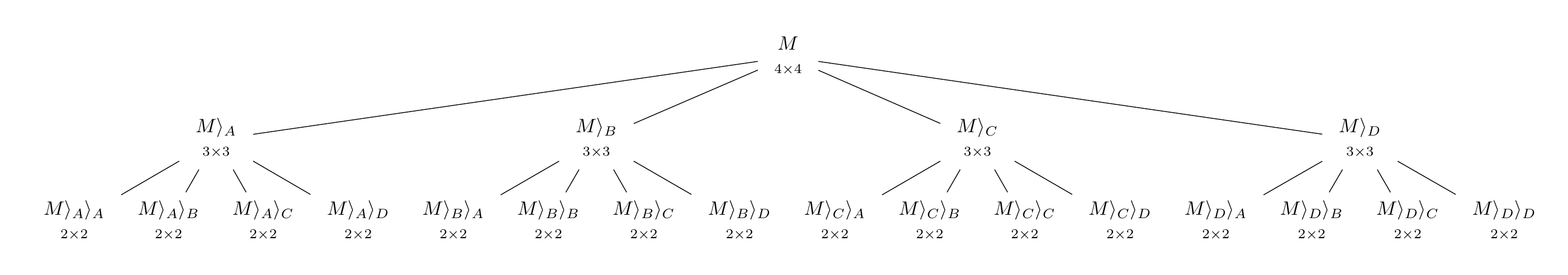}
\caption{Forward recursive procedure}
\label{fig:ForwardProcedure}
\end{figure} 

\begin{figure}[!h]
\centering 
\includegraphics[width=\linewidth, height=60mm]{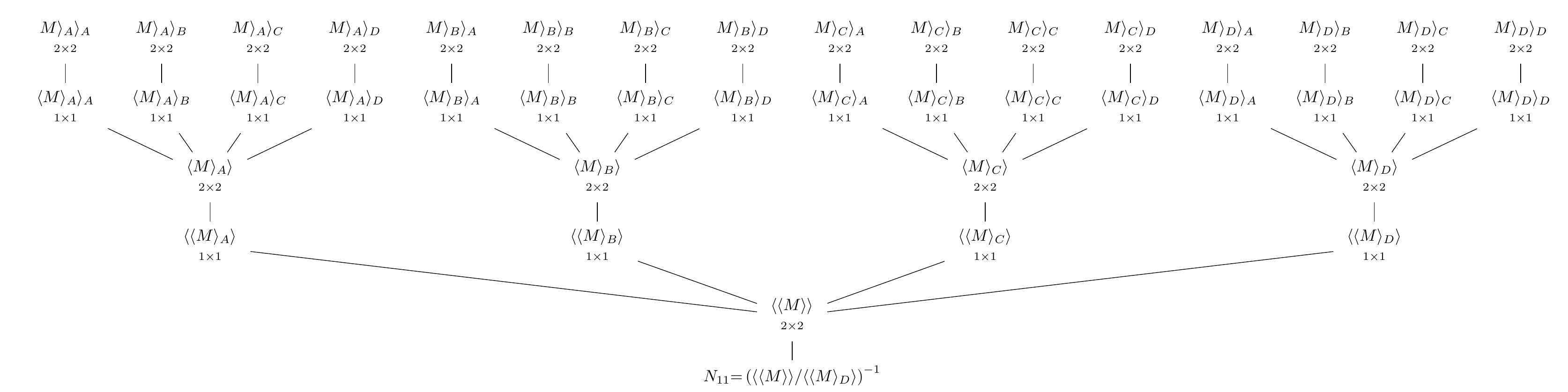}
\caption{Backward recursive procedure}
\label{fig:BackwardProcedure}
\end{figure}

\section{Implementation} \label{sec:implementation}

The Block Recursive Inversion algorithm proposed in Section \ref{sec:recursive_algorithm} has been implemented in C++ language, using the linear algebra library Armadillo \cite{sanderson2016armadillo} for the inversion of the $b \times b$ blocks in the backward procedure.

Due to the recursive structure of the algorithm, it is possible to store only a few $b \times b$ blocks in order to perform the whole computation. This feature allows the inversion of virtually any size of	 matrix to be computed using an amount of memory that scales only with the levels of recursion. In fact, if the backward recursion is computed sequentially, branch by branch, and if we assume that the inversion of a $b \times b$ matrix can be computed in place and consumes only twice the size of a $b \times b$ matrix, it can be shown that each branch needs only $\mathcal{O}(b^2)$ of memory storage in order to complete its computation (see Subsection~\ref{subsec:MemoryAnalysis}).

Additionally, for problems which the matrix to be inverted can be computed element by element, such as in the training of LS-SVMs \cite{an2007fast}, it is not needed to store the original matrix. It can be computed, element by element, from a function of the input data that, although need to be stored, it is often many orders of magnitude smaller. For this application, it might not even be necessary to compute all blocks of the inverse, since the cross-validation procedure only needs the blocks closer to the main diagonal to be performed.

However, since the BRI computes one only block of the inverse, namely the upper left block for computing the complete inverse, it is necessary to run the algorithm $k^2$ times with suitable permutation of rows and columns in order to compute all blocks of the inverse.

\section{Complexity analysis} \label{sec:complexity}

The complexity analysis of the BRI can be built based on the amount Schur complement operations that must be applied to the $2 \times 2$ block matrices in each step of the backward recursive procedure described in Subsection \ref{subsec:Backward}. The forward procedure does not contain arithmetic operations. Consider the following basic facts.

\begin{enumerate}
\item[(a)] The computational cost of the product $XY$ of two blocks $X$ and $Y$ of order $b$ is $\mathcal{O}(b^3)$;
\item[(b)] The computational cost of the inverse $X^{-1}$ of a block $X$ of order $b$ is $\mathcal{O}(b^3)$;
\item[(c)] The definition of Schur complements in (\ref{eq:SchurofD}), (\ref{eq:SchurofC}), (\ref{eq:SchurofB}) and (\ref{eq:SchurofA}) requires each 1 matrix inversion and 2 matrix multiplications for blocks of the order $b$. Hence, the computational cost of a Schur complement operation is $\mathcal{O}(3b^3)$.
\end{enumerate}

\begin{theo} \label{teo:C_BRI}
Let $M \in \mathbb{R}^{m \times m}$ be an invertible $k \times k$ block matrix with blocks of order $b$. Then, the complexity $C_{BRI}(k,b)$ of the BRI algorithm for the inversion of $M$ is
\begin{equation} \label{eq:C_BRI}
C_{BRI}(k,b) = \mathcal{O}(k^2b^34^k).
\end{equation}
\end{theo}

\begin{proof}
According to the basic idea of the BRI algorithm, at step $i$ $(i=1,\dots,k-1)$ of the backward recursive procedure presented in Subsection \ref{subsec:Backward}, we have $(4^{k-1-i})$ $2 \times 2$ block matrices, with blocks of order $b$, to calculate the Schur complement by means of (\ref{eq:SchurofD}), (\ref{eq:SchurofC}), (\ref{eq:SchurofB}) and (\ref{eq:SchurofA}). Thus, by taking property (c) into account at step $i$ $(i=1,\dots,k-1)$ of the backward recursive procedure, we need $3b^34^{k-1-i}$ operations.
Furthermore, at the last step $k-1$, we have to invert the last $2 \times 2$ block matrix to get $N_{11}$, which is one block of $M^{-1}$. Thereby, we need more $k^2 - 1$ executions of the BRI algorithm to get the other blocks of $M^{-1}$. Therefore, the proposed algorithm has a complexity of 
\begin{equation}
 C_{BRI}(k,b) = (\sum\limits_{i=1}^{k-1}3b^34^{k-1-i}+b^3)k^2 = \mathcal{O}(k^2b^34^k).
\end{equation}
\end{proof}

\subsection{Memory cost analysis} \label{subsec:MemoryAnalysis}
According to the proposed implementation of BRI in Section \ref{sec:implementation}, it is possible to store only a few $b \times b$ blocks in order to perform the whole computation due to the recursive structure of the algorithm. 

\begin{theo} \label{teo:MC_BRI}
Let $M \in \mathbb{R}^{m \times m}$ be an invertible $k \times k$ block matrix with blocks of order $b$. Then, the memory cost $MC_{BRI}$ of the BRI algorithm for the inversion of $M$ is
\begin{equation} \label{eq:MC_BRI}
MC_{BRI} = \mathcal{O}(b^2).
\end{equation}
\end{theo}

\begin{proof}
Suppose that it is not needed or that it is not possible to store $M$ in the main memory. Furthermore, each element of $M$ is accessed only to apply the Schur complement operations. So, let the needed memory storage for one block of $M$ of order $b$ be $b^2$ units.
Suppose that the backward recursive procedure defined in Subsection \ref{subsec:Backward} is computed sequentially, branch by branch. Thus, to calculate the Schur complement operation of the $2 \times 2$ block matrix with blocks of order $b$, involved in each branch, it is necessary only $3b^2$ units of memory to store the 2 operators involved and the respective result for each multiplication operation, inversion of block and subtraction operation. So we need $3(b)^2$ units to calculate the Schur complement operation at each branch per level $i$~$(i=1,\dots,k-1)$ of the backward recursion process. Furthermore, $b^2$ more units of memory are necessary to store a block in level $i$ while the Schur complement calculation on a branch of level $i-1$ finishes. Therefore, the proposed algorithm has a memory cost of  
\begin{equation}
 MC_{BRI} = 3(b)^2 + (b)^2 = \mathcal{O}(b^2).
\end{equation}
\end{proof}

It is worth to remark that of the complexity $C_{LU}(k,b)$ and memory cost $MC_{LU}$ of the LU decomposition and numerical inversion of the matrix $M$ is of order 
\begin{equation}
 C_{LU}(k,b) = \mathcal{O}(k^3b^3)
\end{equation}
and 
\begin{equation}
 MC_{LU}(k,b) = \mathcal{O}(k^2b^2).
\end{equation}
Consequently, the LU method is faster than BRI algorithm. On the other hand, the memory cost of the proposed algorithm is much lower than the LU inversion. This fact is verified by the numerical results of Section \ref{sec:results}.

\section{Experimental results} \label{sec:results}

For comparison purposes, we perform experiments with the BRI algorithm and the inverse generated by the calling the Armadillo function {\tt inv($\cdot$)} measuring execution time and memory usage. The {\tt inv($\cdot$)} computes an LU factorization of a general matrix using partial pivoting with row interchanges through integration with LAPACK (Linear Algebra PACKage)~\cite{anderson1999lapack}.

We consider square matrices of order $m$, composed of random entries, chosen from a normal distribution with mean 0.0 and standard derivation 1.0 generated by calling the Armadillo function {\tt randn()}. We performed experiments for matrices with different amounts of blocks ($k \times k$), for BRI, comparing them with unpartitioned matrices of the same size, for LU inversion.

The memory usage measurements were performed with the heap profiler of the Valgrind framework, called Massif \cite{nethercote2007valgrind}. 

The experiments were run on a computer with an AMD Athlon$^{TM}$ II X2 B28 Processor running Linux, with 4GB of RAM.

This section describes the experiments as well as the results obtained. In Subsection \ref{subsec:MemoryUsage}, the measurements of physical memory usage during the execution of BRI and LU inversion is presented. And, in Subsection \ref{subsec:ExecutionTime}, the results on execution time are analyzed.

\subsection{Memory usage} \label{subsec:MemoryUsage}

For both the BRI and the LU inversion, we computed $A^{-1}_{\gamma}$ used in method for cross-validation (CV) of LS-SVM in \cite{an2007fast}, being 

\begin{equation} \label{eq:A}
{A_{\gamma}}_{m \times m}  = \left[\begin{array}{cc}
0 & \textbf{1}^{T}_{n}\\
\textbf{1}_{n} & K_\gamma \\
\end{array}\right]
\end{equation}
with $\textbf{1}^{T}_{n}=[1,1,\dots,1]^T$, $K_\gamma = K + \frac{1}{\gamma} I_n$ and $K_{i,j} = K(x_i,x_j)=\exp (-\frac{||x_i-x_j||^{2}}{2\sigma^{2}})$.

The function $K(\cdot,\cdot)$ is the Radial Basis Function kernel (or Gaussian kernel) and $\{x_i\}_{i=1}^{n}$ is a set of input data. For the tests, this set was composed of random entries, chosen from a normal distribution with mean 0.0 and standard derivation 1.0, generated by the Armadillo command {\tt randn()}.

The memory usage for inversions with respect to the quantity of blocks ($k \times k$), for BRI, and with respect to the unpartitioned matrices, for LU inversion, is shown in Fig. \ref{fig:memory}.
\begin{figure}[htbp]
\centering 
\includegraphics[width=0.90 \textwidth]{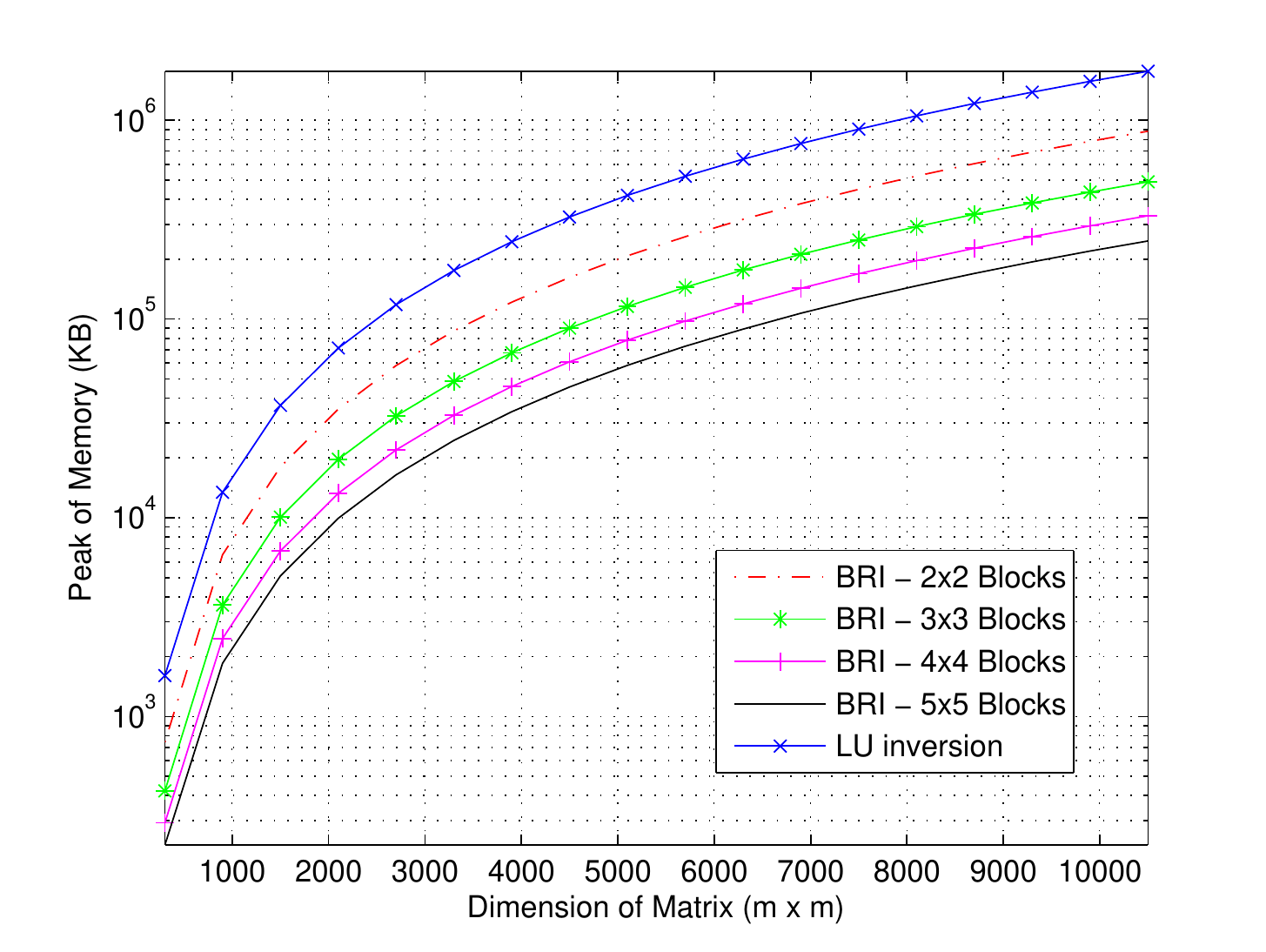}
\caption{Memory usage for the computation of the inverse of $m \times m$ matrices using the LU decomposition and the proposed BRI with different numbers of blocks.}
\label{fig:memory}
\end{figure}

The BRI clearly consumes less physical memory than the LU inversion. Considering an input matrix of the same dimension $m \times m$, as the number of block matrices increases, $k \times k$, the memory usage decreases. This occurs because the dimension $b \times b$ of the blocks decreases, where $b=\frac{m}{k}$, with respect to (\ref{eq:blockmatricesnxn}). In this way, less data is kept in memory during the entire inversion process. Hence, the application of the recursive inversion allows us to consider matrices $A_\gamma$ with much larger orders $m$, than those, which the LU inversion permits.

\subsection{Execution time} \label{subsec:ExecutionTime}

The Fig. \ref{fig:runtime} shows the processing time for the inversion of $m \times m$ matrices for a different number of blocks ($k \times k$), using BRI, and for unpartitioned matrices, using LU. The plots show that the execution time for the BRI is larger than the LU inversion, increasing as we increase the number of blocks that $A_\gamma$ is partitioned. 

\begin{figure}[htbp]
\centering 
\includegraphics[width=0.90 \textwidth]{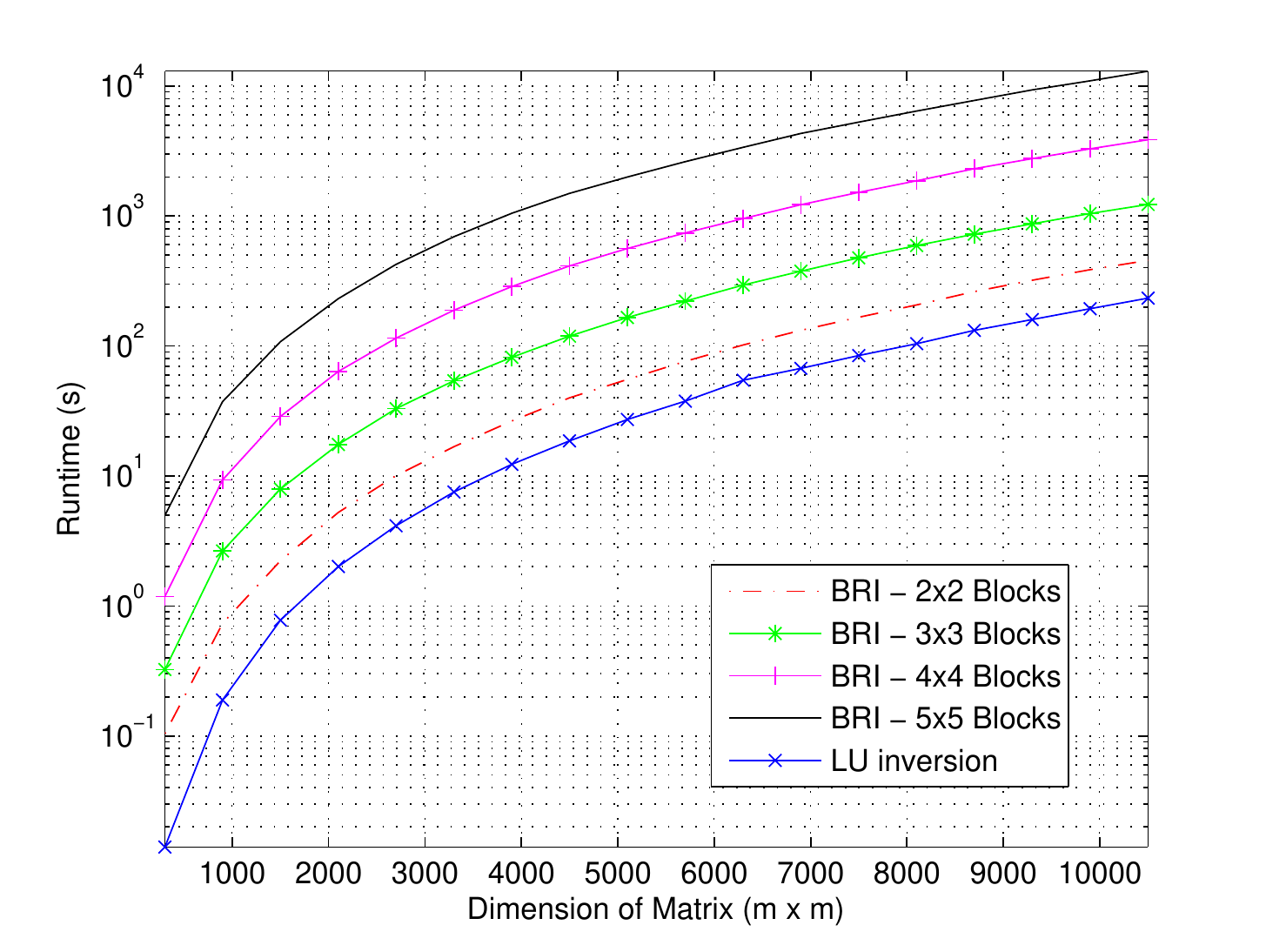}
\caption{CPU time for the computation of the inverse of $m \times m$ matrices using the LU decomposition and the proposed BRI with different numbers of blocks.}
\label{fig:runtime}
\end{figure}

Comparing Figs. \ref{fig:memory} and \ref{fig:runtime}, observe that the use of BRI to calculate the inverse of a $m \times m$ matrix yields a trade-off between memory usage and processing time when varying the number of blocks. This trade-off can be used to allow the inversion of matrices that otherwise would not fit in memory. Moreover, in order to avoid the use of slower memory, the value of $k$ could be chosen such that lower levels of the memory hierarchy would be avoided.

\section{Conclusion} \label{sec:conclusion}

This paper presents a novel recursive algorithm for the inversion of $m \times m$ matrices partitioned in $k-by-k$ blocks, with $k \geq 2$. The Block Recursive Inversion (BRI) algorithm obtains one block of the inverse at a time, resulting from the recursive splitting of the original matrix in $4$ square matrices of an inferior order. The proposed algorithm allows for a reduction of the memory usage which in turns allows inversions of high order matrices which otherwise would exceed the computer memory capacity. 

The experimental results showed that the proposed algorithm consumes much less memory than the LU inversion. This memory usage decrease as the number of block matrices increase, for an input matrix with constant dimension. Increasing the number of blocks also results in an increase of processing time, which characterizes a trade-off between memory usage and processing time. However, since the BRI computes one block of the inverse at a time, the larger processing time could be tackled by using parallel processing to compute all blocks at the same time in an embarrassingly parallel fashion.

It is worth mentioning that the BRI is even more useful in cases that it is only necessary to use parts of the inverse matrix, such as the computing of predicted labels of the cross-validation algorithms in LS-SVMs~\cite{an2007fast}.

Potentially, it might be possible to reduce the computational complexity of the proposed algorithm by reusing some of the computation across branches in the recursion procedure. Parallel implementation and the analysis of the effects of this approach on performance gains due to a possibly better use of the memory hierarchy are also good candidates for future work on this topic. 

Other further extensions of this article are the mathematical proof of this technique, a reformulation of the algorithm with respect to the Schur complement given by (\ref{eq:SchurofA}), the study of its numerical stability,  and possible applications in engineering and computational intelligence. Research on the parallelization of BRI algorithm and its application in cross-validation algorithms for LS-SVM are already in processes.

\section*{Acknowledgements}

This research did not receive any specific grant from funding agencies in the public, commercial, or not-for-profit sectors.

\end{document}